\documentclass[10pt,twoside]{article}
\usepackage{mathrsfs}
\usepackage{amsmath}
\usepackage{amssymb}
\usepackage{fancyhdr}
\usepackage{latexsym}
\usepackage{bbding}
\usepackage{mathrsfs}
\usepackage{wasysym}
\usepackage[notref,notcite]{showkeys}
\usepackage{multicol,graphics}

\newtheorem{theorem}{Theorem}[section]
\newtheorem{lemma}[theorem]{Lemma}
\newtheorem{corollary}[theorem]{Corollary}

\newtheorem{definition}[theorem]{Definition}

\numberwithin{equation}{section}
\newenvironment{proof}[1][Proof]{\noindent\textbf{#1.} }{\hfill $\Box$}
\allowdisplaybreaks

 \makeatletter
\begin{document}
\title{{BKM's criterion for the 3D nematic liquid crystal flows via two velocity components and molecular orientations}\footnote{This paper is partially supported by the National Natural Science Foundation of China (11501453), the Fundamental Research Project of Natural Science in Shaanxi Province--Young Talent Project (2015JQ1004)  and the Fundamental Research Funds for the Central Universities (2014YB031).}}
\author{Jihong Zhao\footnote{Email address: jihzhao@163.com.}\\
[0.2cm] {\small Institute of Applied Mathematics, College of Science, Northwest A\&F
University,}\\
[0.2cm] {\small Yangling, Shaanxi 712100, China}}
\date{}
\maketitle

\begin{abstract}
In this paper we provide a sufficient condition, in terms of the horizontal gradient of two horizontal velocity components and the gradient of liquid crystal molecular orientation field, for the breakdown of local
in time strong solutions to the three-dimensional incompressible nematic liquid
crystal flows. More
precisely, let $T_{*}$ be the maximal existence time of the local
strong solution $(u, d)$, then $T_{*}<+\infty$ if and only if
\begin{align*}
  \int_{0}^{T_{*}}\big(\|\nabla_{h} u^{h}\|_{\dot{B}^{0}_{p,\frac{2p}{3}}}^{q}+\|\nabla d\|_{\dot{B}^{0}_{\infty,\infty}}^{2}\big)dt=\infty\ \ \text{with}
  \ \ \ \frac{3}{p}+\frac{2}{q}=2,\ \ \frac{3}{2}<p\leq\infty,
\end{align*}
where $u^{h}=(u^{1},u^{2})$, $\nabla_{h}=(\partial_{1}, \partial_{2})$.
This result can be regarded as the generalization of the BKM's criterion in \cite{HW12}, and is even new for the three-dimensional incompressible Navier-Stokes equations.

\textbf{Keywords}: Nematic liquid crystal flows; Navier-Stokes equations;
strong solution; BKM's criterion

\textbf{2010 AMS Subject Classification}: 76A15, 35B44, 35Q35
\end{abstract}

\section{Introduction}
The three-dimensional viscous incompressible flow of liquid crystals read:
\begin{equation}\label{eq1.1}
\begin{cases}
  \partial_{t}u+(u\cdot\nabla)u-\nu\Delta u+\nabla\Pi=-\lambda\nabla\cdot(\nabla d \otimes\nabla d) &\text{in}\ \mathbb{R}^{3}\times (0,+\infty),\\
  \nabla\cdot u=0 &\text{in}\ \mathbb{R}^{3}\times (0,+\infty),\\
  \partial_{t}d+(u\cdot\nabla)d=\gamma(\Delta d+|\nabla d|^{2}d) &\text{in}\ \mathbb{R}^{3}\times (0,+\infty),\\
  (u,d)|_{t=0}=(u_{0},d_{0}) &\text{in}\ \mathbb{R}^{3}.
\end{cases}
\end{equation}
Here $u=(u^{1},u^{2},u^{3})\in\mathbb{R}^{3}$ and $\Pi\in\mathbb{R}$ denote the unknown velocity vector field and the scalar pressure field of the fluid motion, respectively, $d=(d^{1},d^{2},d^{3})\in\mathbb{S}^{2}$, the unit sphere in $\mathbb{R}^{3}$, is a unit-vector field that denotes the macroscopic molecular orientation of the nematic liquid crystal material, $\nu$,
$\lambda$, and $\gamma$ are positive constants which represent
viscosity, the competition between kinetic energy and potential
energy, and microscopic elastic relaxation time or the Deborah number for the molecular
orientation field, $u_{0}$ and $d_{0}$ are the initial datum of $u$ and $d$,
and $u_{0}$ satisfies $\nabla\cdot u_{0}=0$ in the distributional
sense. Here and in what follows, we denote by $\nabla$ and $\nabla\cdot$ the gradient operator $(\partial_{1}, \partial_{2}, \partial_{3})$ and the divergence operator, and
\begin{equation*}
  (u\cdot\nabla)v=(\sum_{i=1}^{3}u^{i}\partial_{i}v^{j})_{j=1}^{3},\ \ \
  \nabla\cdot u=\sum_{i=1}^{3}\partial_{i}u^{i}\ \ \ \text{for}\ \ \ u,v\in \mathbb{R}^{3}.
\end{equation*}
The term $\nabla d\otimes\nabla d$ in the stress
tensor represents the anisotropic feature of the system, which is
the $3\times 3$ matrix whose $(i,j)$-th entry is given by
$\partial_{i}d\cdot
\partial_{j}d$ ($1\leq i,j\leq 3$).
Moreover, it is easy to verify that
\begin{equation}\label{eq1.2}
  \nabla\cdot(\nabla d\otimes \nabla d)=\nabla(\frac{|\nabla d|^{2}}{2})+\Delta d\cdot\nabla
  d.
\end{equation}
Since the exact values of the constants $\nu$, $\lambda$, and $\gamma$ do not play a special role in our discussion,
for simplicity, we henceforth assume that $\nu=\lambda=\gamma=1$ throughout this paper.

The hydrodynamic theory of liquid crystals was developed by Ericksen and Leslie during the period of 1958 through 1968 (see \cite{E61}, \cite{L68}, \cite{L79}). The system \eqref{eq1.1} is a simplified version of Ericksen-Leslie model, which reduces to the
Ossen-Frank model in the static case, it is a macroscopic continuum description
of the time evolution of the materials under the influence of both the flow field $u$, and the macroscopic description
of the microscopic orientation configurations $d$ of rod-like liquid crystals.  Mathematically, the system
\eqref{eq1.1} is a strongly nonlinear coupled system between the non-homogeneous incompressible Navier-Stokes equations ( the case $d$ equals a constant unit vector, see \eqref{eq1.3} below) and the transported heat flow
of harmonic maps (the case $u\equiv0$) into $\mathbb{S}^{2}$, therefore its mathematical analysis is full of challenges, see the
expository paper of Lin and Wang \cite{LW14} for more details.

In 1989, Lin \cite{L89} first derived a simplified Ericksen-Leslie system modeling liquid crystal flows, and then made some important analytic studies of this simplified system in \cite{LL95} and \cite{LL96} with collaborator Liu. More precisely, they considered the Ginzburg-Landau
approximation or the so-called orientation with variable degrees in the terminology of Ericksen, i.e.,
the Dirichlet energy $\int_{\mathbb{R}^{3}}\frac{|\nabla d|^{2}}{2}dx$ for $d:\mathbb{R}^{3}\rightarrow\mathbb{S}^{2}$ is replaced by the Ginzburg-Landau
energy $\int_{\mathbb{R}^{3}}\frac{|\nabla d|^{2}}{2}+\frac{(1-|d|^{2})^{2}}{4\varepsilon^{2}}dx$ ($\varepsilon>0$) for  $d:\mathbb{R}^{3}\rightarrow\mathbb{R}^{3}$, thus the third equation of \eqref{eq1.1} is replaced by
\begin{equation*}
  \partial_{t}d+(u\cdot\nabla)d=\gamma(\Delta d+\frac{1}{\varepsilon^{2}}(1-|d|^{2})d).
\end{equation*}
Under this
simplification, Lin and Liu proved in \cite{LL95} the global existence of a unique strong solution in dimension two and in dimension three under the large
viscosity $\nu$, moreover, in
\cite{LL96}, they proved the existence of suitable weak solutions and the one-dimensional space-time
Hausdorff measure of the singular set of suitable weak solutions
is zero, analogous to the celebrated partial regularity theorem by Caffarelli-Kohn-Nirenberg \cite{CKN82} for the three-dimensional incompressible Navier-Stokes equations. Note that the regularity and uniqueness of global weak solutions to the three-dimensional simplified system are still open problems, we refer the readers to see \cite{LZC11}, \cite{WLY15}, \cite{ZLC13} and \cite{ZF10} for some recent results concerning the global existence of strong solutions with small initial data, regularity and uniqueness criteria of weak solutions to the simplified Ericksen-Leslie system.

Recently, many attempts on rigorous mathematical analysis of the system \eqref{eq1.1} were established.  Lin, Lin and Wang
\cite{LLW10} studied the Dirichlet initial-boundary problem of \eqref{eq1.1}, they proved the results that for any
$(u_{0}, d_{0})\in L^{2}(\Omega, \mathbb{R}^{2})\times H^{1}(\Omega, \mathbb{S}^{2})$ ($\Omega\subset\mathbb{R}^{2}$ is a bounded smooth domain), there exists a global Leray-Hopf weak
solution $(u,d)$ that is smooth away from at most finitely many singularity times, see also Hong \cite{H11}. The uniqueness of such
weak solution was subsequently obtained by Lin and Wang \cite{LW10}
and Xu and Zhang \cite{XZ12}. Later on, many authors studied the global well-posedness of the system \eqref{eq1.1} with small initial data in various functional
spaces, for instance, the Lebesgue space $L^{3}_{uloc}(\mathbb{R}^{2})$
by Hineman and Wang in \cite{HW13}, the
critical and subcritical Sobolev spaces by Lin and Ding \cite{LD12}, the Besov spaces of either
positive or negative regularity indices by Li and Wang
\cite{LW12} and Hao and Liu \cite{HL14}. In particular, Wang \cite{W11} addressed both local and global well-posedness with rough initial data $(u_{0}, d_{0})\in BMO^{-1}\times BMO$, where $BMO$ is the space of  \textit{Bounded Mean
Oscillation}. We refer the readers to see \cite{DW13}, \cite{DW14}, \cite{L15}, \cite{LX15},  \cite{LZZ15},  \cite{LZ13} and \cite{WX14} for further studies to the system \eqref{eq1.1}.

Note that, when the liquid crystal molecular orientation $d$ equals to a constant unit-vector, the system
\eqref{eq1.1} reduces into the three-dimensional incompressible Navier-Stokes equations:
\begin{equation}\label{eq1.3}
\begin{cases}
  \partial_{t} u+(u\cdot\nabla)u-\nu\Delta
  u+\nabla \Pi=0, \ \ &\text{in}\ \mathbb{R}^{3}\times (0,+\infty),\\
  \nabla\cdot u=0,\ \ &\text{in}\ \mathbb{R}^{3}\times (0,+\infty),\\
  u|_{t=0}=u_0,\ \ &\text{in}\ \mathbb{R}^{3}.
\end{cases}
\end{equation}
In his pioneering work \cite{L34}, Leray proved the global existence of strong solutions in dimension two and in dimension three under the large viscosity $\nu$ or small initial data. However, for general initial data, in dimension three, whether the corresponding local strong solution can be extended to the global one is the challenge open problem.
In the celebrated work, Beale, Kato and Majda \cite{BKM84} showed that, by using the logarithmic Sobolev
inequality, if the smooth solution $u$ blows up at the time
$t=T_{*}$, then
\begin{equation}\label{eq1.4}
  \int_{0}^{T_{*}}\|\omega(\cdot,
  t)\|_{L^{\infty}}\;dt=\infty.
\end{equation}
Here $\omega=\nabla\times u$ is the vorticity. Subsequently, Kozono and Taniuchi \cite{KT00} and Kozono, Ogawa and
Taniuchi \cite{KOT02} refined the BKM's criterion \eqref{eq1.4}
to
\begin{equation}\label{eq1.5}
  \int_{0}^{T_{*}}\|\omega(\cdot, t)\|_{BMO}\;dt=\infty\
  \ \text{and}\ \ \int_{0}^{T_{*}}\|\omega(\cdot, t)\|_{\dot{B}^{0}_{\infty,
  \infty}}\;dt=\infty,
\end{equation}
respectively, where $\dot{B}^{0}_{\infty, \infty}(\mathbb{R}^{3})$ is the homogeneous
Besov space. Recently, Dong and Zhang \cite{DZ10} established the BKM's criterion to the equations \eqref{eq1.3} that if
\begin{equation}\label{eq1.6}
  \int_{0}^{T}\|\nabla_{h}u^{h}(\cdot, t)\|_{\dot{B}^{0}_{\infty,
  \infty}}\;dt<\infty,
\end{equation}
then the strong solution $u$ can be extended beyond the time $T$.

On the other hand, when the velocity field $u$ is identically vanishing, the system
\eqref{eq1.1} becomes to the heat flow of harmonic maps. In
\cite{W08}, Wang established  a Serrin type regularity criterion,
which implies that if the solution $d$ blows up at time $T_{*}$,
then
\begin{equation}\label{eq1.7}
  \sup_{0\leq t< T_{*}}\|\nabla d(\cdot, t)\|_{L^{n}}=\infty.
\end{equation}
Motivated by the conditions \eqref{eq1.4} and \eqref{eq1.7}, Huang
and Wang \cite{HW12} established a BKM type blow-up criterion for
the system \eqref{eq1.1}, more precisely, they characterized the first finite singular
time $T_{*}$ as that if $T_{*}<\infty$, then
\begin{equation}\label{eq1.8}
  \int_{0}^{T_{*}}\big(\|\omega(\cdot, t)\|_{L^{\infty}}+\|\nabla d(\cdot, t)\|_{L^{\infty}}^{2}\big)\;dt=\infty.
\end{equation}
Recently, Liu and the author in this paper improved the condition \eqref{eq1.8} in \cite{LZ14} that the smooth solution $(u,d)$ of
\eqref{eq1.1} blows up at the time $T_{*}$ if and only if
\begin{align}\label{eq1.9}
  \int_{0}^{T_{*}}\frac{\|\omega(\cdot, t)\|_{\dot{B}^{0}_{\infty,\infty}}+\|\nabla d(\cdot, t)\|_{\dot{B}^{0}_{\infty,\infty}}^{2}}
  {\sqrt{1+\text{ln}(e+\|\omega(\cdot, t)\|_{\dot{B}^{0}_{\infty,\infty}}
  +\|\nabla d(\cdot, t)\|_{\dot{B}^{0}_{\infty,\infty}})}}\text{d}t=\infty.
\end{align}

Before we state our main result. we recall that local existence of smooth solutions of the system \eqref{eq1.1} has
been announced in \cite{HW12}. For a given unit vector
$\bar{d}\in\mathbb{S}^{2}$ and $s>0$, we set
\begin{equation*}
  H^{s}_{\bar{d}}(\mathbb{R}^{3}, \mathbb{S}^{2}):=\Big\{d: d-\bar{d}\in H^{s}(\mathbb{R}^{3}, \mathbb{R}^{3}),
  |d|=1\ \text{a.e.}\ \text{in} \ \mathbb{R}^{3}\Big\}.
\end{equation*}
It follows from \cite{HW12} that if the initial velocity
$u_{0}\in H^{s}(\mathbb{R}^{3},\mathbb{R}^{3})$ with
$\nabla\cdot u_{0}=0$ and $d_{0}\in
H^{s+1}_{\bar{d}}(\mathbb{R}^{3}, \mathbb{S}^{2})$  for $s\geq 3$,
then there exists a positive time $T_{*}$ depending only on
$\|(u_{0}, \nabla d_{0})\|_{H^{s}}$ such that the system
\eqref{eq1.1} has a unique smooth solution $(u, d)$ in
$\mathbb{R}^{3}\times[0,T_{*})$ satisfying
\begin{equation}\label{eq1.10}
\begin{cases}
  u\in C([0,T], H^{s}(\mathbb{R}^{3},\mathbb{R}^{3}))\cap
  C^{1}([0,T], H^{s-2}(\mathbb{R}^{3}, \mathbb{R}^{3})), \\
  d\in C([0,T], H^{s+1}_{\bar{d}}(\mathbb{R}^{3},
  \mathbb{S}^{2}))\cap C^{1}([0,T], H^{s-1}_{\bar{d}}(\mathbb{R}^{3},
  \mathbb{S}^{2})).
\end{cases}
\end{equation}
for any $0<T<T_{*}$. In this paper, we aim at extending the BKM's criterion \eqref{eq1.6} for the Navier-Stokes equations \eqref{eq1.3} and  establishing a similar BKM's criterion to the system \eqref{eq1.1} in terms of the horizontal gradient of two horizontal velocity components and the gradient of liquid crystal molecular orientation field in the framework of the homogeneous Besov space.

\begin{theorem}\label{th1.1}
Let $u_{0}\in H^{3}(\mathbb{R}^{3},\mathbb{R}^{3})$ with
$\nabla\cdot u_{0}=0$ and $d_{0}\in
H^{4}_{\bar{d}}(\mathbb{R}^{3},\mathbb{S}^{2})$, and let $T_{*}>0$
be the maximal existence time such that the system \eqref{eq1.1} has
a unique solution $(u,d)$ satisfying \eqref{eq1.10} for any
$0<T<T_{*}$. If $T_{*}<+\infty$, then
\begin{align}\label{eq1.11}
  \int_{0}^{T_{*}}\big(\|\nabla_{h} u^{h}\|_{\dot{B}^{0}_{p,\frac{2p}{3}}}^{q}+\|\nabla d\|_{\dot{B}^{0}_{\infty,\infty}}^{2}\big)dt=\infty\ \ \text{with}
  \ \ \ \frac{3}{p}+\frac{2}{q}=2,\ \ \frac{3}{2}<p\leq\infty,
\end{align}
where $\nabla_{h}=(\partial_{1}, \partial_{2})$ and $u^{h}=(u^{1}, u^{2})$. In particular,
\begin{align*}
  \limsup_{t\uparrow T_{*}}\big(\|\nabla_{h} u^{h}\|_{\dot{B}^{0}_{p,\frac{2p}{3}}}+\|\nabla d\|_{\dot{B}^{0}_{\infty,\infty}}\big)=\infty,\  \ \frac{3}{2}<p\leq\infty.
\end{align*}
\end{theorem}

\noindent\textbf{Remark 1.1} It is clearly that the blow-up criterion \eqref{eq1.11} is
an improvement of \eqref{eq1.8} in \cite{HW12} due to the embedding $L^{\infty}(\mathbb{R}^{3})\hookrightarrow \dot{B}^{0}_{\infty,\infty}(\mathbb{R}^{3})$ and the fact $\|\omega\|_{\dot{B}^{0}_{\infty,\infty}}\cong\|\nabla u\|_{\dot{B}^{0}_{\infty,\infty}}$ .

\medskip

As a byproduct of Theorem \ref{th1.1}, we obtain a corresponding BKM's blow-up criterion for the three-dimensional Navier-Stokes equations.
\begin{corollary}\label{co1.2}
Let $u_{0}\in H^{1}(\mathbb{R}^{3},\mathbb{R}^{3})$ with
$\nabla\cdot u_{0}=0$, and let $T_{*}>0$
be the maximal existence time such that the system \eqref{eq1.3} has
a unique solution $u$ satisfying $u\in C([0,T], H^{1}(\mathbb{R}^{3}))$ for any
$0<T<T_{*}$. If $T_{*}<+\infty$, then
\begin{align}\label{eq1.12}
   \int_{0}^{T_{*}}\|\nabla_{h} u^{h}\|_{\dot{B}^{0}_{p,\frac{2p}{3}}}^{q}dt=\infty\ \ \text{with}
  \ \ \ \frac{3}{p}+\frac{2}{q}=2,\ \ \frac{3}{2}<p\leq\infty.
\end{align}
In particular,
\begin{align*}
  \limsup_{t\uparrow T_{*}}\|\nabla_{h} u^{h}\|_{\dot{B}^{0}_{p,\frac{2p}{3}}}=\infty,\  \ \frac{3}{2}<p\leq\infty.
\end{align*}
\end{corollary}

\noindent\textbf{Remark 1.2} When $p=\infty$, the condition \eqref{eq1.12} becomes $\int_{0}^{T_{*}}\|\nabla_{h} u^{h}\|_{\dot{B}^{0}_{\infty,\infty}}dt=\infty$, thus Corollary
\ref{co1.2} can be regarded as an extension of the blow up criterion \eqref{eq1.6}
in \cite{DZ10}.

The plan of the paper is arranged as follows. In Section 2, we recall the
Littlewood-Paley dyadic decomposition theory, the definition of
the homogeneous Besov spaces, and then state some important inequalities. In Section
3, we present the proof of Theorem \ref{th1.1}. Throughout the
paper, $C$ denotes a constant and may change from line to line;
$\|\cdot\|_{X}$ stands for the norm of the Banach space $X$, and we say that a vector $u=(u^{1}, u^{2}, u^{3})$ belongs to a
function space $X$ if $u^{j}\in X$ holds for every $j=1,2,3$ and we
put $\|u\|_{X}\overset{\operatorname{def}}{=}\underset{1\leq j\leq3}{\max}\|u^{j}\|_{X}$.

\section{Preliminaries}

In this section we shall recall some preliminaries on the
Littlewood-Paley decomposition theory and the definition of the homogeneous Besov spaces, and state some important inequalities in these
functional spaces. For the details, see \cite{BCD11} and \cite{T83}.

We first introduce the Littlewood-Paley dyadic decomposition theory. Let
$\mathcal{S}(\mathbb{R}^{3})$ be the Schwartz class of rapidly
decreasing function and $\mathcal{S}'(\mathbb{R}^{3})$ be its dual.
Given $f\in\mathcal{S}(\mathbb{R}^{3})$, its Fourier transformation
$\mathcal{F}(f)$ or $\widehat{f}$ is
defined by
$$
  \mathcal{F}(f)(\xi)=\widehat{f}(\xi)\overset{\operatorname{def}}{=}\frac{1}{(2\pi)^{\frac{3}{2}}}\int_{\mathbb{R}^{3}}f(x)e^{-ix\cdot\xi}dx.
$$
Let $\mathcal{B}=\{\xi\in\mathbb{R}^{3},\ |\xi|\leq\frac{4}{3}\}$
and $\mathcal{C}=\{\xi\in\mathbb{R}^{3},\ \frac{3}{4}\leq
|\xi|\leq\frac{8}{3}\}$. Choose two nonnegative smooth radial
functions $\chi, \varphi\in\mathcal{S}(\mathbb{R}^{3})$, valued in the interval $[0,1]$, and
supported on $\mathcal{B}$ and $\mathcal{C}$,  respectively, such that
\begin{align*}
  \chi(\xi)+\sum_{j\geq0}\varphi(2^{-j}\xi)=1, \ \ \xi\in\mathbb{R}^{3},\\
  \sum_{j\in\mathbb{Z}}\varphi(2^{-j}\xi)=1, \ \
  \xi\in\mathbb{R}^{3}\backslash\{0\}.
\end{align*}
Writing $\varphi_{j}(\xi)=\varphi(2^{-j}\xi)$. Then it is easy to
verify that
$\operatorname{supp}\varphi_{j}\cap\operatorname{supp}\varphi_{k}=\varnothing$
if $|j-k|\geq2$. Let $h=\mathcal{F}^{-1}\varphi$ and
$\tilde{h}=\mathcal{F}^{-1}\chi$, where $\mathcal{F}^{-1}$ is the
inverse Fourier transform, i.e.,
$$
  \mathcal{F}^{-1}(f)(x)\overset{\operatorname{def}}{=}\frac{1}{(2\pi)^{\frac{3}{2}}}\int_{\mathbb{R}^{3}}f(\xi)e^{i\xi\cdot x}d\xi.
$$
The homogeneous dyadic blocks
$\Delta_{j}$ and $S_{j}$ are defined for all $j\in \mathbb{Z}$ by
\begin{align*}
  \Delta_{j}f\overset{\operatorname{def}}{=}\varphi(2^{-j}\mathrm{D})f=2^{3j}\int_{\mathbb{R}^{3}}h(2^{j}y)f(x-y)dy,\\
  S_{j}f\overset{\operatorname{def}}{=}\chi(2^{-j}\mathrm{D})f=2^{3j}\int_{\mathbb{R}^{3}}\tilde{h}(2^{j}y)f(x-y)dy.
\end{align*}
Here $\mathrm{D}=(\mathrm{D}_1, \mathrm{D}_2,\mathrm{D}_3)$ and
$\mathrm{D}_j=i^{-1}\partial_{x_j}$ ($i^{2}=-1$, $j=1,2,3$). By telescoping the series, we have the following (formal) Littlewood-Paley decomposition
\begin{equation*}
  f=\sum_{j\in\mathbb{Z}}\Delta_{j}f \ \ \text{for}\ \
  f\in\mathcal{S}_{h}'(\mathbb{R}^{3}),
\end{equation*}
where $\mathcal{S}_{h}'(\mathbb{R}^{3})$ is the space of temperate
distributions $f$ such that
\begin{equation*}
\lim_{j\rightarrow-\infty}S_{j}f=0\ \ \text{in}\ \
\mathcal{S}'(\mathbb{R}^{3}).
\end{equation*}
We remark here that $\mathcal{S}_{h}'(\mathbb{R}^{3})$ can be
identified by the quotient space of
$\mathcal{S}'(\mathbb{R}^{3})/\mathcal{P}(\mathbb{R}^{3})$ with the
polynomial space $\mathcal{P}(\mathbb{R}^{3})$.

\medskip

Now let us recall
the definition of the homogeneous Besov spaces.

\begin{definition}\label{de2.1}
Let $s\in \mathbb{R}$, $p,r\in[1,\infty]$. The homogeneous Besov
space $\dot{B}^{s}_{p,r}(\mathbb{R}^{3})$ is defined by
\begin{equation*}
  \dot{B}^{s}_{p,r}(\mathbb{R}^{3})\overset{\operatorname{def}}{=}\Big\{f\in \mathcal{S}'_{h}(\mathbb{R}^{3}):\ \
  \|f\|_{\dot{B}^{s}_{p,r}}<\infty\Big\},
\end{equation*}
where
\begin{equation*}
  \|f\|_{\dot{B}^{s}_{p,r}}\overset{\operatorname{def}}{=} \begin{cases} \big(\sum_{j\in\mathbb{Z}}2^{jsr}\|\Delta_{j}f\|_{L^{p}}^{r}\big)^{\frac{1}{r}}
  \ \ &\text{for}\ \ 1\leq r<\infty,\\
  \sup_{j\in\mathbb{Z}}2^{js}\|\Delta_{j}f\|_{L^{p}}\ \
  &\text{for}\ \
  r=\infty.
 \end{cases}
\end{equation*}
\end{definition}

We emphasize here that the homogeneous Hilbert space
$\dot{H}^{s}(\mathbb{R}^{3})$ can be identified with the homogeneous Besov space $\dot{B}^{s}_{2,2}(\mathbb{R}^{3})$,
which is endowed with the equivalent norm
$\|f\|_{\dot{H}^{s}}=\|(-\Delta)^{\frac{s}{2}}f\|_{L^{2}}$.

\medskip

Next, we state two key lemmas used in the proof of Theorem \ref{th1.1}. The first one is the classical Bernstein inequalities, which can be easily derived from the Young inequality.

\begin{lemma}\label{le2.2} {\em (\cite{BCD11})}
For any nonnegative integer $k$ and any
couple of real numbers $(p,q)$ with $1\leq p\le q\leq \infty$, we
have
\begin{equation}\label{eq2.1}
   \sup_{|\alpha|=k}\|\partial^{\alpha}\Delta_{j}f\|_{L^{q}}\leq
   C2^{jk+3j(\frac{1}{p}-\frac{1}{q})}\|\Delta_{j}f\|_{L^{p}},
\end{equation}
where $C$ is a positive constant independent of $f$ and $j$.
\end{lemma}

The second one is an interpolation inequality due to Meyer, Gerard
and Oru \cite{MGO96}.

\begin{lemma}\label{le2.3}{\em (\cite{MGO96})}
Let $1<p<q<\infty$, and let $s=\beta(\frac{q}{p}-1)>0$. Then there
exists a constant $C$ depending only on $\beta$, $p$ and $q$
such that the estimate
\begin{equation}\label{eq2.2}
  \|f\|_{L^{q}}\leq
  C\|f\|_{\dot{B}^{-\beta}_{\infty,\infty}}^{1-\frac{p}{q}}\|(-\Delta)^{\frac{s}{2}}f\|_{L^{p}}^{\frac{p}{q}}
\end{equation}
holds for all $f\in
\dot{H}^{s}_{p}(\mathbb{R}^{3})\cap\dot{B}^{-\beta}_{\infty,\infty}(\mathbb{R}^{3})$.
\end{lemma}

In fact, we shall use the following particular form
of \eqref{eq2.2} by taking $s=\beta=1$, $p=2$ and $q=4$:
\begin{equation}\label{eq2.3}
  \|f\|_{L^{4}}\leq
  C\|f\|_{\dot{B}^{-1}_{\infty,\infty}}^{\frac{1}{2}}\|(-\Delta)^{\frac{1}{2}}f\|_{L^{2}}^{\frac{1}{2}},
\end{equation}
which the proof can be found in \cite{GG11}.

\section{The proof of Theorem \ref{th1.1}}

We prove Theorem \ref{th1.1} by contradiction.  Assume that \eqref{eq1.11} were not true. Then there exists a positive constant
$M_{0}$ such that
\begin{align}\label{eq3.1}
  \int_{0}^{T_{*}}\big(\|\nabla_{h} u^{h}(\cdot, t)\|_{\dot{B}^{0}_{p,\frac{2p}{3}}}^{q}+\|\nabla d(\cdot, t)\|_{\dot{B}^{0}_{\infty,\infty}}^{2}\big)dt\leq M_{0}.
\end{align}
Based on the BKM's criterion \eqref{eq1.8} in \cite{HW12}, it suffices to show that if under the condition  \eqref{eq3.1}, there exists a constant $C$ such that
\begin{equation}\label{eq3.2}
  \int_{0}^{T_{*}}\big(\|\omega(\cdot, t)\|_{L^{\infty}}+\|\nabla d(\cdot, t)\|_{L^{\infty}}^{2}\big)\;dt\leq C,
\end{equation}
which is enough to guarantee the extension of
strong solution ($u$, $d$) beyond the time $ T_{*}$ by \cite{HW12}. That is to say,
$[0,T_{*})$ is not a maximal existence interval, we get the desired
contradiction.

\medskip

In order to prove \eqref{eq3.2}, we first aim at establishing $L^{2}$
estimates for the vorticity $\omega$ and $\Delta d$.

\begin{lemma}\label{le3.1}
Assume that $u_{0}\in H^{3}(\mathbb{R}^{3},\mathbb{R}^{3})$ with $\nabla\cdot u_{0}=0$, $d_{0}\in
H^{4}_{\bar{d}}(\mathbb{R}^{3},\mathbb{S}^{2})$. Let $(u,d)$ be the
corresponding local smooth solution to the system
\eqref{eq1.1} on $[0,T)$ for some $0<T<\infty$. If
\begin{align}\label{eq3.3}
 \int_{0}^{T}\big(\|\nabla_{h} u^{h}(\cdot, t)\|_{\dot{B}^{0}_{p,\frac{2p}{3}}}^{q}+\|\nabla d(\cdot, t)\|_{\dot{B}^{0}_{\infty,\infty}}^{2}\big)dt=M_{0}<\infty,
\end{align}
then we have
\begin{align}\label{eq3.4}
   \sup_{0\leq t\leq T}(\|\omega(\cdot,t)\|_{L^{2}}^{2}&+\|\Delta d(\cdot,t)\|_{L^{2}}^{2})+\int_{0}^{T}\big(\|\nabla
  \omega(\cdot,t)\|_{L^{2}}^{2}+\|\nabla\Delta d(\cdot,t)\|_{L^{2}}^{2}\big)dt\leq C_{0},
\end{align}
where $C_{0}$ is a constant depending only on $\|u_{0}\|_{H^{1}}$, $\|d_{0}\|_{H^{2}}$, $T$ and $M_{0}$.
\end{lemma}
\begin{proof}
Firstly, multiplying the first equation
of \eqref{eq1.1} by $u$, integrating over $\mathbb{R}^{3}$, after
integration by parts, one has
\begin{equation}\label{eq3.5}
\frac{1}{2}\frac{d}{dt}\|u\|_{L^{2}}^{2}+\|\nabla u\|_{L^{2}}^{2}=-\int_{\mathbb{R}^{3}}\Delta
d\cdot\nabla d\cdot udx,
\end{equation}
where we have used the facts \eqref{eq1.2} and
$\nabla\cdot u=0$ imply that
$$
  \int_{\mathbb{R}^{3}}\nabla(\frac{|\nabla d|^{2}}{2})\cdot udx=-\int_{\mathbb{R}^{3}}(\frac{|\nabla d|^{2}}{2})(\nabla\cdot u)dx=0.
$$
Multiplying the third equation of \eqref{eq1.1} by $-\Delta d$,
integrating over $\mathbb{R}^{3}$, and using the facts $|d|=1$ and
$|\nabla d|^{2}=-d\cdot\Delta d$, one obtains that
\begin{align}\label{eq3.6}
  \frac{1}{2}\frac{d}{dt}\|\nabla d\|_{L^{2}}^{2}+\|\Delta d\|_{L^{2}}^{2}-\int_{\mathbb{R}^{3}}(u\cdot\nabla) d\cdot\Delta ddx
  &=-\int_{\mathbb{R}^{3}}|\nabla d|^{2}d\cdot\Delta d dx\nonumber\\
  &=\int_{\mathbb{R}^{3}}|d\cdot\Delta d|^{2}dx\leq\int_{\mathbb{R}^{3}}|\Delta
  d|^{2}dx.
\end{align}
Adding the above estimates \eqref{eq3.5} and \eqref{eq3.6} together,
we have
\begin{equation*}
  \frac{d}{dt}\big(\|u\|_{L^{2}}^{2}+\|\nabla d\|_{L^{2}}^{2}\big)\leq 0,
\end{equation*}
which yields the following basic energy inequality:
\begin{equation}\label{eq3.7}
  \sup_{t\geq0}\big(\|u(\cdot, t)\|_{L^{2}}^{2}+\|\nabla d(\cdot, t)\|_{L^{2}}^{2}\big)\leq
  \|u_{0}\|_{L^{2}}^{2}+\|\nabla d_{0}\|_{L^{2}}^{2}.
\end{equation}

Secondly, taking $\nabla\times$  on the first equation
of \eqref{eq1.1}, we see that
\begin{equation}\label{eq3.8}
  \partial_{t}\omega+(u\cdot\nabla)\omega-\Delta \omega=(\omega\cdot\nabla)u-\nabla\times(\Delta d\cdot \nabla d),
\end{equation}
where we have used the facts \eqref{eq1.2} and  $\nabla\times\nabla(\frac{|\nabla d|^{2}}{2})=0$.
Multiplying \eqref{eq3.8} by $w$, and integrating over $\mathbb{R}^{3}$, one has
\begin{align}\label{eq3.9}
  \frac{1}{2}\frac{d}{dt}\|\omega\|_{L^{2}}^{2}+\|\nabla \omega\|_{L^{2}}^{2}&=\int_{\mathbb{R}^{3}}(\omega\cdot\nabla)u\cdot \omega dx+\int_{\mathbb{R}^{3}}(\Delta d\cdot\nabla d)\cdot(\nabla\times \omega)dx\nonumber\\
  &\overset{\operatorname{def}}{=}I_{1}+I_{2}.
\end{align}
Note that
\begin{align*}
  I_{1}=\sum_{i,j=1}^{3}\int_{\mathbb{R}^{3}}\omega_{i}\partial_{i}u^{j}\omega_{j} dx
  &=\sum_{i,j=1}^{2}\int_{\mathbb{R}^{3}}\omega_{i}\partial_{i}u^{j}\omega_{j} dx
  +\sum_{i=1}^{2}\int_{\mathbb{R}^{3}}\omega_{i}\partial_{i}u^{3}\omega_{3} dx\\
  &+\sum_{j=1}^{2}\int_{\mathbb{R}^{3}}\omega_{3}\partial_{3}u^{j}\omega_{j} dx
  +\int_{\mathbb{R}^{3}}\omega_{3}\partial_{3}u^{3}\omega_{3} dx.
\end{align*}
Thus, by using the facts
\begin{align*}
  \partial_{3}u^{3}=-(\partial_{1}u^{1}+\partial_{2}u^{2}), \ \ \ \omega_{3}=\partial_{1}u^{2}-\partial_{2}u^{1},
\end{align*}
we can easily see that
\begin{equation}\label{eq3.10}
  I_{1}\leq \int_{\mathbb{R}^{3}}|\nabla_{h}u^{h}||\omega|^{2}dx.
\end{equation}
Applying the Littlewood-Paley dyadic decomposition for $\nabla_{h}u^{h}$, i.e.,
\begin{equation*}
  \nabla_{h}u^{h}=\sum_{j<-N}\Delta_{j}\nabla_{h}u^{h}+\sum_{j=-N}^{N}\Delta_{j}\nabla_{h}u^{h}+\sum_{j>N}\Delta_{j}\nabla_{h}u^{h}.
\end{equation*}
Then, we have
\begin{align}\label{eq3.11}
  I_{1}&\leq \sum_{j<-N}\int_{\mathbb{R}^{3}}|\Delta_{j}\nabla_{h}u^{h}||\omega|^{2}dx+\sum_{j=-N}^{N}\int_{\mathbb{R}^{3}}|\Delta_{j}\nabla_{h}u^{h}||\omega|^{2}dx
  +\sum_{j>N}\int_{\mathbb{R}^{3}}|\Delta_{j}\nabla_{h}u^{h}||\omega|^{2}dx\nonumber\\
  &\overset{\operatorname{def}}{=}I_{11}+I_{12}+I_{13}.
\end{align}
Now we estimate the terms $I_{1j}$ ($j=1,2,3$) one by one. For $I_{11}$, by using the H\"{o}lder inequality, the Bernstein inequality \eqref{eq2.1} and the fact $\|\nabla u\|_{L^{2}}\leq C\|\omega\|_{L^{2}}$ (cf. \cite{KT00}), one has
\begin{align}\label{eq3.12}
  I_{11}&=\sum_{j<-N}\int_{\mathbb{R}^{3}}|\Delta_{j}\nabla_{h}u^{h}||\omega|^{2}dx\nonumber\\
  &\leq C\sum_{j<-N}\|\Delta_{j}\nabla_{h}u^{h}\|_{L^{\infty}}\|\omega\|_{L^{2}}^{2}\nonumber\\
  &\leq C \sum_{j<-N}2^{\frac{3}{2}j}\|\Delta_{j}\nabla_{h}u^{h}\|_{L^{2}}\|\omega\|_{L^{2}}^{2}\nonumber\\
  &\leq C \Big(\sum_{j<-N}2^{3j}\Big)^{\frac{1}{2}}\Big(\sum_{j<-N}\|\Delta_{j}\nabla_{h}u^{h}\|_{L^{2}}^{2}\Big)^{\frac{1}{2}}\|\omega\|_{L^{2}}^{2}\nonumber\\
  &\leq C2^{-\frac{3}{2}N}\|\nabla u\|_{L^{2}}\|\omega\|_{L^{2}}^{2}\nonumber\\
  &\leq C2^{-\frac{3}{2}N}\|\omega\|_{L^{2}}^{3}.
\end{align}
For $I_{12}$, let $p'$ be a conjugate index of $p$, i.e., $\frac{1}{p}+\frac{1}{p'}=1$. Then applying the H\"{o}lder inequality and Definition \ref{de2.1},  we get
\begin{align}\label{eq3.13}
  I_{12}&=\sum_{j=-N}^{N}\int_{\mathbb{R}^{3}}|\Delta_{j}\nabla_{h}u^{h}||\omega|^{2}dx\nonumber\\
  &\leq C\sum_{j=-N}^{N}\|\Delta_{j}\nabla_{h}u^{h}\|_{L^{p}}\|\omega\|_{L^{2p'}}^{2}\nonumber\\
  &\leq C\Big(\sum_{j=-N}^{N}1^{\frac{2p}{2p-3}}\Big)^{\frac{2p-3}{2p}}
  \Big(\sum_{j=-N}^{N}\|\Delta_{j}\nabla_{h}u^{h}\|_{L^{p}}^{\frac{2p}{3}}\Big)^{\frac{3}{2p}}\|\omega\|_{L^{2p'}}^{2}\nonumber\\
  &\leq CN^{\frac{2p-3}{2p}}\|\nabla_{h}u^{h}\|_{\dot{B}^{0}_{p,\frac{2p}{3}}}\|\omega\|_{L^{2p'}}^{2}\nonumber\\
  &\leq CN^{\frac{2p-3}{2p}}\|\nabla_{h}u^{h}\|_{\dot{B}^{0}_{p,\frac{2p}{3}}}\|\omega\|_{L^{2}}^{2(1-\frac{3}{2p})}\|\nabla\omega\|_{L^{2}}^{\frac{3}{p}}\nonumber\\
  &\leq \frac{1}{8}\|\nabla\omega\|_{L^{2}}^{2}+CN\|\nabla_{h}u^{h}\|_{\dot{B}^{0}_{p,\frac{2p}{3}}}^{\frac{2p}{2p-3}}\|\omega\|_{L^{2}}^{2},
\end{align}
where we have used the
following Gagliardo-Nirenberg inequality:
\begin{equation*}
  \|w\|_{L^{\frac{2p}{p-1}}}\leq \|w\|_{L^{2}}^{1-\frac{3}{2p}}\|
  \nabla w\|_{L^{2}}^{\frac{3}{2p}}\ \ \text{for}\ \ p\geq\frac{3}{2}.
\end{equation*}
For $I_{13}$, from the equality
\begin{align*}
  \nabla\times \omega=\nabla\times(\nabla\times u)=\nabla(\nabla\cdot u)-\Delta u
  =-\Delta u,
\end{align*}
one can easily derived from the integration by parts that
\begin{align*}
  \|\nabla^{2} u\|_{L^{2}}=\|\Delta u\|_{L^{2}}\leq C\|\nabla \omega\|_{L^{2}}.
\end{align*}
Hence, based on the above inequality, applying the H\"{o}lder inequality and the Bernstein inequality \eqref{eq2.1}
again, $I_{13}$ can be estimated as follows:
\begin{align}\label{eq3.14}
  I_{13}&=\sum_{j>N}\int_{\mathbb{R}^{3}}|\Delta_{j}\nabla_{h}u^{h}||\omega|^{2}dx\nonumber\\
  &\leq C\sum_{j>N}\|\Delta_{j}\nabla_{h}u^{h}\|_{L^{3}}\|\omega\|_{L^{2}}\|w\|_{L^{6}}\nonumber\\
  &\leq C \sum_{j>N}2^{\frac{j}{2}}\|\Delta_{j}\nabla_{h}u^{h}\|_{L^{2}}\|\omega\|_{L^{2}}\|\nabla \omega\|_{L^{2}}\nonumber\\
  &\leq C \Big(\sum_{j>N}2^{-j}\Big)^{\frac{1}{2}}\Big(\sum_{j>N}2^{2j}\|\Delta_{j}\nabla_{h}u^{h}\|_{L^{2}}^{2}\Big)^{\frac{1}{2}}\|\omega\|_{L^{2}}\|\nabla \omega\|_{L^{2}}\nonumber\\
  &\leq C2^{-\frac{N}{2}}\|\omega\|_{L^{2}}\|\nabla^{2} u\|_{L^{2}}\|\nabla \omega\|_{L^{2}}\nonumber\\
  &\leq C2^{-\frac{N}{2}}\|\omega\|_{L^{2}}\|\nabla \omega\|_{L^{2}}^{2}.
\end{align}
As for $I_{2}$, by using the H\"{o}lder's inequality, the interpolation inequality  \eqref{eq2.3} with $f=\Delta d$ and $|d|=1$, we get
\begin{align}\label{eq3.15}
  I_{2}&=\int_{\mathbb{R}^{3}}(\Delta d\cdot\nabla d)\cdot(\nabla\times \omega)dx\nonumber\\
  &\leq C\||\nabla d||\Delta d|\|_{L^{2}}\|\nabla \omega\|_{L^{2}}\nonumber\\
  &\leq C \|\nabla d\|_{L^{4}}\|\Delta d\|_{L^{4}}\|\nabla \omega\|_{L^{2}}\nonumber\\
  &\leq \frac{1}{8}\|\nabla \omega\|_{L^{2}}^{2}+C\|\nabla d\|_{L^{4}}^{2}\|\Delta d\|_{L^{4}}^{2}\nonumber\\
  &\leq \frac{1}{8}\|\nabla \omega\|_{L^{2}}^{2}+C\|d\|_{L^{\infty}}\|\Delta d\|_{L^{2}}\|\nabla d\|_{\dot{B}^{0}_{\infty,\infty}}\|\nabla\Delta d\|_{L^{2}}\nonumber\\
  &\leq \frac{1}{8}\|\nabla \omega\|_{L^{2}}^{2}+\frac{1}{8}\|\nabla\Delta d\|_{L^{2}}^{2}+C\|\nabla d\|_{\dot{B}^{0}_{\infty,\infty}}^{2}\|\Delta d\|_{L^{2}}^{2},
\end{align}
where we have used the
following Gagliardo-Nirenberg inequality with $q=2$:
\begin{equation*}
  \|\nabla f\|_{L^{2q}}\leq \|f\|_{L^{\infty}}^{\frac{1}{2}}\|
  \Delta f\|_{L^{q}}^{\frac{1}{2}}\ \ \text{for}\ \ q>\frac{3}{2}.
\end{equation*}
Inserting estimates \eqref{eq3.12}--\eqref{eq3.15} into
\eqref{eq3.9}, it follows that
\begin{align}\label{eq3.16}
  \frac{d}{dt}\|\omega \|_{L^{2}}^{2}+\frac{3}{2}\|\nabla \omega\|_{L^{2}}^{2}&\leq \frac{1}{4}\|\nabla\Delta d\|_{L^{2}}^{2}+C2^{-\frac{3}{2}N}\|\omega\|_{L^{2}}^{3}+CN\|\nabla_{h}u^{h}\|_{\dot{B}^{0}_{p,\frac{2p}{3}}}^{\frac{2p}{2p-3}}\|\omega\|_{L^{2}}^{2}\nonumber\\
  &+C2^{-\frac{N}{2}}\|\omega\|_{L^{2}}\|\nabla \omega\|_{L^{2}}^{2}+C\|\nabla d\|_{\dot{B}^{0}_{\infty,\infty}}^{2}\|\Delta d\|_{L^{2}}^{2}.
\end{align}

For the estimate of $\Delta d$, taking $\Delta$ on the third equation of \eqref{eq1.1}, multiplying $\Delta d$
and integrating over $\mathbb{R}^{3}$, one has
\begin{align}\label{eq3.17}
  \frac{1}{2}\frac{d}{dt}\|\Delta d\|_{L^{2}}^{2}+\|\nabla \Delta d\|_{L^{2}}^{2}
  &=-\int_{\mathbb{R}^{3}}\Delta((u\cdot\nabla) d)\cdot\Delta d dx+\int_{\mathbb{R}^{3}}\Delta(|\nabla d|^{2}d)\cdot\Delta d dx\nonumber\\
  &\overset{\operatorname{def}}{=}J_{1}+J_{2}.
\end{align}
Note that the fact $\nabla\cdot u=0$ implies that
\begin{align*}
  \int_{\mathbb{R}^{3}}(u\cdot\nabla)\Delta d\cdot\Delta d dx=\frac{1}{2}\int_{\mathbb{R}^{3}}u\cdot\nabla(|\Delta d|^{2})dx=\frac{1}{2}\int_{\mathbb{R}^{3}}(\nabla\cdot u)(|\Delta d|^{2})dx=0.
\end{align*}
Therefore, applying \eqref{eq2.3} with $f=\nabla^{2}d$, we can estimate $J_{1}$ as follows:
\begin{align}\label{eq3.18}
  J_{1}&=-\int_{\mathbb{R}^{3}}\Delta((u\cdot\nabla) d)\cdot\Delta d dx\nonumber\\
  &\leq \int_{\mathbb{R}^{3}}|\Delta u||\nabla d||\Delta d| dx+2\int_{\mathbb{R}^{3}}|\nabla u||\nabla^{2}d||\Delta d| dx\nonumber\\
  &\leq C\Big(\|\Delta u\|_{L^{2}}\|\nabla d\|_{L^{4}}\|\Delta d\|_{L^{4}}+\|\nabla u\|_{L^{2}}\|\nabla^{2}d\|_{L^{4}}\|\Delta d\|_{L^{4}}\Big)\nonumber\\
  &\leq C\Big(\|\nabla \omega\|_{L^{2}}\|d\|_{L^{\infty}}^{\frac{1}{2}}\|\Delta d\|_{L^{2}}^{\frac{1}{2}}\|\nabla d\|_{\dot{B}^{0}_{\infty,\infty}}^{\frac{1}{2}}\|\nabla\Delta d\|_{L^{2}}^{\frac{1}{2}}
  +\|\omega\|_{L^{2}}\|\nabla d\|_{\dot{B}^{0}_{\infty,\infty}}\|\nabla\Delta d\|_{L^{2}}\Big)\nonumber\\
  &\leq \frac{1}{8}\|\nabla \omega\|_{L^{2}}^{2}+\frac{1}{8}\|\nabla \Delta d\|_{L^{2}}^{2}
  +C\|\nabla d\|_{\dot{B}^{0}_{\infty,\infty}}^{2}\big(\|\omega\|_{L^{2}}^{2}+\|\Delta d\|_{L^{2}}^{2}\big),
\end{align}
where we have used the Gagliardo-Nirenberg inequalities:
\begin{align*}
  \|\nabla^{2}d\|_{L^{4}}\leq \|\nabla d\|_{L^{\infty}}^{\frac{1}{2}}\|\nabla \Delta d\|_{L^{2}}^{\frac{1}{2}}, \ \ \
  \|\Delta d\|_{L^{4}}\leq \|\nabla d\|_{L^{\infty}}^{\frac{1}{2}}\|\nabla \Delta d\|_{L^{2}}^{\frac{1}{2}}.
\end{align*}
For $J_{2}$, after integration by parts, by using the H\"{o}lder inequality, the
Young inequality, the inequality \eqref{eq2.3} with $f=\nabla^{2}d$ and the fact $|d|=1$, we obtain
\begin{align}\label{eq3.19}
  J_{2}&\leq \Big{|}\int_{\mathbb{R}^{3}}\Delta(d|\nabla d|^{2})\cdot\Delta d dx\Big{|}
  =\Big{|}\int_{\mathbb{R}^{3}}\nabla(d|\nabla d|^{2})\cdot\nabla\Delta d dx\Big{|}\nonumber\\
  &\leq\Big{|}\int_{\mathbb{R}^{3}}\nabla d|\nabla d|^{2}\cdot\nabla\Delta d dx\Big{|}
  +\Big{|}\int_{\mathbb{R}^{3}} d\nabla(|\nabla d|^{2})\cdot\nabla\Delta d dx\Big{|}\nonumber\\
  &\leq\Big{|}\int_{\mathbb{R}^{3}}\nabla(\nabla d|\nabla d|^{2})\cdot\Delta d dx\Big{|}
  +\Big{|}\int_{\mathbb{R}^{3}} d\nabla(|\nabla d|^{2})\cdot\nabla\Delta d dx\Big{|}\nonumber\\
  &\leq C\Big{|}\int_{\mathbb{R}^{3}}|\nabla d|^{2}|\nabla^{2} d|^{2} dx\Big{|}
  +C\Big{|}\int_{\mathbb{R}^{3}} |\nabla d||\nabla^{2}d||\nabla\Delta d| dx\Big{|}\nonumber\\
  &\leq C\|\nabla d\|_{L^{4}}^{2}\|\nabla^{2}d\|_{L^{4}}^{2}
  +C\|\nabla d\|_{L^{4}}\|\nabla^{2} d\|_{L^{4}}\|\nabla\Delta d\|_{L^{2}}\nonumber\\
  &\leq \frac{1}{8}\|\nabla\Delta d\|_{L^{2}}^{2}+C\|\nabla d\|_{L^{4}}^{2}\|\nabla^{2} d\|_{L^{4}}^{2}\nonumber\\
  &\leq \frac{1}{8}\|\nabla\Delta d\|_{L^{2}}^{2}+C\|d\|_{L^{\infty}}\|\Delta d\|_{L^{2}}
  \|\nabla d\|_{\dot{B}^{0}_{\infty,\infty}}\|\nabla\Delta d\|_{L^{2}}\nonumber\\
  &\leq \frac{1}{4}\|\nabla\Delta d\|_{L^{2}}^{2}+C\|\nabla d\|_{\dot{B}^{0}_{\infty,\infty}}^{2}\|\Delta d\|_{L^{2}}^{2}.
\end{align}
Combining the estimates \eqref{eq3.18} and \eqref{eq3.19}, we obtain that
\begin{align}\label{eq3.20}
  \frac{d}{dt}\|\Delta d\|_{L^{2}}^{2}+\frac{5}{4}\|\nabla \Delta d\|_{L^{2}}^{2}
  \leq \frac{1}{4}\|\nabla \omega\|_{L^{2}}^{2}
  +C\|\nabla d\|_{\dot{B}^{0}_{\infty,\infty}}^{2}\big(\|\omega\|_{L^{2}}^{2}+\|\Delta d\|_{L^{2}}^{2}\big).
\end{align}

Finally, putting \eqref{eq3.16} and \eqref{eq3.20} together, we obtain
\begin{align}\label{eq3.21}
  \frac{d}{dt}(\|\omega \|_{L^{2}}^{2}&+\|\Delta d\|_{L^{2}}^{2})+\frac{5}{4}\|\nabla \omega\|_{L^{2}}^{2}+\|\nabla \Delta d\|_{L^{2}}^{2}
  \leq C2^{-\frac{3}{2}N}\|\omega\|_{L^{2}}^{3}
  +CN\|\nabla_{h}u^{h}\|_{\dot{B}^{0}_{p,\frac{2p}{3}}}^{q}\|\omega\|_{L^{2}}^{2}\nonumber\\
  &+C2^{-\frac{N}{2}}\|\omega\|_{L^{2}}\|\nabla \omega\|_{L^{2}}^{2}
  +C\|\nabla d\|_{\dot{B}^{0}_{\infty,\infty}}^{2}\big(\|\omega\|_{L^{2}}^{2}+\|\Delta d\|_{L^{2}}^{2}\big),
\end{align}
where $q=\frac{2p}{2p-3}$ satisfying $\frac{3}{p}+\frac{2}{q}=2$. Now we choose $N$ such that
\begin{equation*}
   C2^{-\frac{N}{2}}\|\omega\|_{L^{2}}\leq \frac{1}{4},
\end{equation*}
i.e.,
\begin{equation*}
   N\geq \frac{\ln(\|\omega\|_{L^{2}}^{2}+e)+\ln C}{\ln2}+2.
\end{equation*}
Thus, the inequality \eqref{eq3.21} yields that
\begin{align}\label{eq3.22}
  \frac{d}{dt}(\|\omega \|_{L^{2}}^{2}&+\|\Delta d\|_{L^{2}}^{2})+\|\nabla \omega\|_{L^{2}}^{2}+\|\nabla \Delta d\|_{L^{2}}^{2}
  \leq
  C(1+\|\nabla_{h}u^{h}\|_{\dot{B}^{0}_{p,\frac{2p}{3}}}^{q}+\|\nabla d\|_{\dot{B}^{0}_{\infty,\infty}}^{2})\nonumber\\
  &\times(\|\omega\|_{L^{2}}^{2}+\|\Delta d\|_{L^{2}}^{2})\ln(\|\omega\|_{L^{2}}^{2}+\|\Delta d\|_{L^{2}}^{2}+e).
\end{align}
Setting $H(t)=\|\omega(t)\|_{L^{2}}^{2}+\|\Delta d(t)\|_{L^{2}}^{2}$. It follows from \eqref{eq3.22} that
\begin{align}\label{eq3.23}
  \frac{d}{dt}\ln(e+H(t))\leq
  C(1+\|\nabla_{h}u^{h}\|_{\dot{B}^{0}_{p,\frac{2p}{3}}}^{q}+\|\nabla d\|_{\dot{B}^{0}_{\infty,\infty}}^{2})\ln(e+H(t)).
\end{align}
Integrating \eqref{eq3.23} in time from $0$ to $t$ with $0<t<T$, we have
\begin{align*}
  H(t)\leq (e+H(0))\exp\Big\{\exp\big\{Ct+C\int_{0}^{t}(\|\nabla_{h}u^{h}\|_{\dot{B}^{0}_{p,\frac{2p}{3}}}^{q}+\|\nabla d\|_{\dot{B}^{0}_{\infty,\infty}}^{2})ds\big\}\Big\}.
\end{align*}
Namely,
\begin{align}\label{eq3.24}
  \sup_{0\leq t<T}(\|\omega\|_{L^{2}}^{2}&+\|\Delta d\|_{L^{2}}^{2})\leq (e+\|\nabla u_{0}\|_{L^{2}}^{2}+\|\Delta d_{0}\|_{L^{2}}^{2})\nonumber\\
  &\times\exp\Big\{\exp\big\{CT+C\int_{0}^{T}(\|\nabla_{h}u^{h}\|_{\dot{B}^{0}_{p,\frac{2p}{3}}}^{q}+\|\nabla d\|_{\dot{B}^{0}_{\infty,\infty}}^{2})ds\big\}\Big\}.
\end{align}
We complete the proof of  Lemma \ref{le3.1}.
\end{proof}

\medskip

Next,  in order to derive the desired result \eqref{eq3.2}, let us establish the following result.
\begin{lemma}\label{le3.2}
Under the assumptions of  Lemma \ref{le3.1}, we have
\begin{equation}\label{eq3.25}
  \int_{0}^{T}\big(\|\nabla\Delta u(\cdot,t)\|_{L^{2}}+\|\Delta^{2} d(\cdot,t)\|_{L^{2}}\big)dt\leq C,
\end{equation}
where $C$ is constant depending only on $\|u_{0}\|_{H^{2}}$, $\|d_{0}\|_{H^{3}}$, $T$ and $M_{0}$.
\end{lemma}
\begin{proof}
Applying $\Delta$ on the first
equation of \eqref{eq1.1}, then multiplying the resulting equation by $\Delta u$ and
integrating over $\mathbb{R}^{3}$, after integration by parts, one
obtains
\begin{align}\label{eq3.26}
  \frac{1}{2}\frac{d}{dt}\|\Delta u\|_{L^{2}}^{2}+\|\nabla\Delta u\|_{L^{2}}^{2}&=-\int_{\mathbb{R}^{3}}\Delta((u\cdot\nabla) u)\cdot\Delta udx
  -\int_{\mathbb{R}^{3}}\Delta(\Delta d\cdot\nabla d)\cdot\Delta
  udx,
\end{align}
where we have used the fact that $\operatorname{div} u=0$ implies
that $\int_{\mathbb{R}^{3}}\Delta\nabla(\frac{|\nabla
d|^{2}}{2})\cdot\Delta u\text{d}x=0$. Applying \eqref{eq3.4} in Lemma \ref{le3.1},  the right hand side of \eqref{eq3.26} can be bounded as follows:
\begin{align*}
  -\int_{\mathbb{R}^{3}}\Delta((u\cdot\nabla)
  u)\cdot\Delta u dx&\leq \frac{1}{8}\|\nabla \Delta u\|_{L^{2}}^{2} +C\big(\|(\nabla u\cdot\nabla) u\|_{L^{2}}^{2}+\|(u\cdot\nabla)\nabla u\|_{L^{2}}^{2}\big)\nonumber\\
  &\leq \frac{1}{8}\|\nabla \Delta u\|_{L^{2}}^{2} +C\big(\|\nabla u\|_{L^{4}}^{4}+\|u\|_{L^{6}}^{2}\|\Delta u\|_{L^{3}}^{2}\big)\nonumber\\
   &\leq \frac{1}{8}\|\nabla \Delta u\|_{L^{2}}^{2} +C\big(\|\nabla u\|_{L^{2}}^{\frac{5}{2}}\|\nabla \Delta u\|_{L^{2}}^{\frac{3}{2}}+\|\nabla u\|_{L^{2}}^{2}\|\Delta u\|_{L^{2}}\|\nabla\Delta u\|_{L^{2}}\big)\nonumber\\
   &\leq\frac{1}{8}\|\nabla \Delta u\|_{L^{2}}^{2} +C(\|\Delta u\|_{L^{2}}^{2}+1);
  \end{align*}
\begin{align*}
  -\int_{\mathbb{R}^{3}}\Delta(\Delta d\nabla d)\cdot\Delta u dx&=\int_{\mathbb{R}^{3}}\nabla(\Delta d\nabla d)\cdot\nabla\Delta u dx\nonumber\\
  &\leq \frac{1}{8}\|\nabla\Delta u\|_{L^{2}}^{2}+C\big(\|\nabla\Delta d\nabla d\|_{L^{2}}^{2}+\|\Delta d\nabla ^{2}d\|_{L^{2}}^{2}\big)\nonumber\\
  &\leq \frac{1}{8}\|\nabla\Delta u\|_{L^{2}}^{2}+C\big(\|\nabla\Delta d\|_{L^{3}}^{2}\|\nabla d\|_{L^{6}}^{2}+\|\Delta d\|_{L^{4}}^{4}\big)\nonumber\\
  &\leq \frac{1}{8}\|\nabla\Delta u\|_{L^{2}}^{2}+C\big(\|\Delta d\|_{L^{2}}^{2}\|\nabla\Delta d\|_{L^{2}}\|\Delta^{2} d\|_{L^{2}}+\|\Delta d\|_{L^{2}}^{\frac{5}{2}}\|\Delta^{2} d\|_{L^{2}}^{\frac{3}{2}}\big)\nonumber\\
  &\leq \frac{1}{8}\|\nabla\Delta u\|_{L^{2}}^{2}+ \frac{1}{6}\|\Delta^{2} d\|_{L^{2}}^{2}+C\big(\|\nabla\Delta d\|_{L^{2}}^{2}+1\big).
\end{align*}
Hence, we infer from \eqref{eq3.26} that
\begin{align}\label{eq3.27}
  \frac{d}{dt}\|\Delta
  u\|_{L^{2}}^{2}+\frac{3}{2}\|\nabla\Delta u\|_{L^{2}}^{2}\leq\frac{1}{3}\|\Delta^{2} d\|_{L^{2}}^{2}+C\big(\|\Delta u\|_{L^{2}}^{2}+\|\nabla\Delta d\|_{L^{2}}^{2}+1\big).
\end{align}
Taking $\nabla\Delta$ on the third equation of \eqref{eq1.1},
then multiplying the resulting equation by $\nabla\Delta d$ and integrating over $\mathbb{R}^{3}$, one
obtains
\begin{align}\label{eq3.28}
  \frac{1}{2}\frac{d}{dt}\|\nabla\Delta d\|_{L^{2}}^{2}&+\|\Delta^{2} d\|_{L^{2}}^{2}
  =-\int_{\mathbb{R}^{3}}\nabla\Delta((u\cdot\nabla) d)\cdot\nabla\Delta d dx
  +\int_{\mathbb{R}^{3}}\nabla\Delta(|\nabla d|^{2}d)\cdot\nabla\Delta d
  dx.
\end{align}
Applying the Leibniz's rule, the facts  $|d|=1$ and \eqref{eq3.4} again, the right hand side of \eqref{eq3.28} can be bounded as follows:
\begin{align*}
   -&\int_{\mathbb{R}^{3}}\nabla\Delta((u\cdot\nabla) d)\cdot\nabla\Delta d dx
   =\int_{\mathbb{R}^{3}}\Delta((u\cdot\nabla) d)\cdot\Delta^{2} d dx\nonumber\\
   &\leq \frac{1}{8}\|\Delta^{2} d\|_{L^{2}}^{2} +C\big(\|(\Delta u\cdot\nabla) d\|_{L^{2}}^{2}+\|(\nabla u\cdot\nabla)\nabla d\|_{L^{2}}^{2}+\|(u\cdot\nabla)\Delta d\|_{L^{2}}^{2}\big)\nonumber\\
   &\leq \frac{1}{8}\|\Delta^{2} d\|_{L^{2}}^{2} +C\big(\|\Delta u\|_{L^{3}}^{2}\|\nabla d\|_{L^{6}}^{2}+\|\nabla u\|_{L^{4}}^{2}\|\Delta d\|_{L^{4}}^{2}+\|u\|_{L^{6}}^{2}\|\nabla\Delta d\|_{L^{3}}^{2}\big)\nonumber\\
   &\leq \frac{1}{8}\|\Delta^{2} d\|_{L^{2}}^{2} +C\big(\|\Delta u\|_{L^{2}}\|\nabla\Delta u\|_{L^{2}}+\|\nabla u\|_{L^{2}}^{\frac{5}{4}}\|\nabla \Delta u\|_{L^{2}}^{\frac{3}{4}}\|\Delta d\|_{L^{2}}^{\frac{5}{4}}\|\Delta^{2} d\|_{L^{2}}^{\frac{3}{4}}+\|\nabla\Delta d\|_{L^{2}}\|\Delta^{2} d\|_{L^{2}}\big)\nonumber\\
   &\leq \frac{1}{6}\|\Delta^{2} d\|_{L^{2}}^{2}+\frac{1}{4}\|\nabla\Delta u\|_{L^{2}}^{2}+C\big(\|\Delta u\|_{L^{2}}^{2}+\|\nabla\Delta d\|_{L^{2}}^{2}+1\big);
\end{align*}
\begin{align*}
  \int_{\mathbb{R}^{3}}&\nabla\Delta(|\nabla d|^{2}d)\cdot\nabla\Delta d
  dx=-\int_{\mathbb{R}^{3}}\Delta(|\nabla d|^{2}d)\cdot\Delta^{2}d
  dx\nonumber\\
  &=-\int_{\mathbb{R}^{3}}\big[\Delta(|\nabla d|^{2}) d+2\nabla(|\nabla d|^{2})\nabla d
   +|\nabla d|^{2}\Delta d\big]\cdot\Delta^{2}ddx\nonumber\\
  &\leq C\big(\|\nabla
  d\|_{L^{6}}\|\nabla\Delta d\|_{L^{3}}+\|\Delta d\|_{L^{4}}^{2}+\|\nabla d\|_{L^{6}}^{2}\|\Delta d\|_{L^{3}}\big)
  \|\Delta^{2} d\|_{L^{2}}\nonumber\\
  &\leq C\big(\|\Delta d\|_{L^{2}}^{\frac{1}{2}}\|\Delta^{2} d\|_{L^{2}}^{\frac{1}{2}}
  +\|\Delta d\|_{L^{2}}^{\frac{5}{4}}\|\Delta^{2}
  d\|_{L^{2}}^{\frac{3}{4}}+\|\Delta d\|_{L^{2}}^{\frac{5}{2}}\|\nabla\Delta
  d\|_{L^{2}}^{\frac{1}{2}}\big)\|\Delta^{2} d\|_{L^{2}}\nonumber\\
  &\leq\frac{1}{6}\|\Delta^{2} d\|_{L^{2}}^{2}+C\big(\|\nabla\Delta
  d\|_{L^{2}}^{2}+1\big).
\end{align*}
Hence, one obtains from \eqref{eq3.28} that
\begin{align}\label{eq3.29}
  \frac{d}{dt}\|\nabla \Delta d\|_{L^{2}}^{2}+\frac{4}{3}\|\Delta^{2} d\|_{L^{2}}^{2}\leq\frac{1}{2}\|\nabla \Delta u\|_{L^{2}}^{2}+C\big(\|\Delta u\|_{L^{2}}^{2}+\|\nabla\Delta d\|_{L^{2}}^{2}+1\big).
\end{align}
We conclude from \eqref{eq3.27} and \eqref{eq3.29} that
\begin{align}\label{eq3.30}
  \frac{d}{dt}\big(\|\Delta
  u\|_{L^{2}}^{2}+\|\nabla\Delta d\|_{L^{2}}^{2}\big)+\big(\|\nabla\Delta u\|_{L^{2}}^{2}+\|\Delta^{2}d\|_{L^{2}}^{2}\big)
  \leq C\big(\|\Delta u\|_{L^{2}}^{2}+\|\nabla\Delta d\|_{L^{2}}^{2}+1\big),
\end{align}
which gives us to the desired result \eqref{eq3.25} by the Gronwall's inequality. The proof of Lemma \ref{le3.2} is complete.
\end{proof}

Finally, by using the Sobolev embedding  $H^{2}(\mathbb{R}^{3})\rightarrow L^{\infty}(\mathbb{R}^{3})$, \eqref{eq3.25} leads to the BKM's criterion \eqref{eq3.2} immediately.
We complete the proof of  Theorem \ref{th1.1}.

\end{document}